\documentclass[11pt]{article}
\usepackage{amsthm, amsmath, amssymb, amsfonts, url, booktabs, tikz, setspace, fancyhdr, bm}
\usepackage[hidelinks]{hyperref}
\usepackage{geometry}
\usepackage{hyperref, enumerate}
\usepackage[shortlabels]{enumitem}
\usepackage[babel]{microtype}
\usepackage[english]{babel}
\usepackage[capitalise]{cleveref}
\usepackage{comment}
\usepackage{bbm}
\usepackage{csquotes}
\usepackage{mathabx}

\counterwithin{figure}{section}

\newtheorem{theorem}{Theorem}[section]
\newtheorem{prop}[theorem]{Proposition}

\newtheorem{cor}[theorem]{Corollary}
\newtheorem{conj}[theorem]{Conjecture}
\newtheorem{claim}[theorem]{Claim}

\theoremstyle{definition}

\theoremstyle{remark}

\newenvironment{poc}{\begin{proof}[Proof of claim]}{\end{proof}}

\newcommand{\C}[1]{{\protect\mathcal{#1}}}

\usepackage{todonotes}

\title{Stability through non-shadows}
\author{
Jun Gao\thanks{Extremal Combinatorics and Probability Group (ECOPRO), Institute for Basic Science (IBS), Daejeon, South Korea. Emails: {\texttt \{jungao, hongliu, zixiangxu\}@ibs.re.kr}. Supported by IBS-R029-C4.}
\and 
Hong Liu\footnotemark[1]
\and
Zixiang Xu\footnotemark[1]
}

\begin{document}

\maketitle
\begin{abstract}
 We study families $\mathcal{F}\subseteq 2^{[n]}$ with restricted intersections and prove a conjecture of Snevily in a stronger form for large $n$. We also obtain stability results for Kleitman's isodiametric inequality and families with bounded set-wise differences. Our proofs introduce a new twist to the classical linear algebra method, harnessing the non-shadows of $\mathcal{F}$, which may be of independent interest. 
\end{abstract}

\section{Introduction}




\subsection{Restricted intersections}
The celebrated theorem of Erd\H{o}s, Ko and Rado~\cite{1961EKR} states that when $n\geq 2k$, every $k$-uniform intersecting family $\mathcal{F}\subseteq 2^{[n]}$ has size at most $\binom{n-1}{k-1}$. We can view the Erd\H{o}s-Ko-Rado theorem as a result on families with restricted intersections where the empty intersection is forbidden. 
A family $\mathcal{F}\subseteq 2^{[n]}$ is \emph{Sperner} if for any distinct $A,B\in \C F$, $A\not\subseteq B$. Another cornerstone result in extremal set theory is Sperner's theorem~\cite{1928Sperner} from 1928, stating that any Sperner family in $2^{[n]}$ has size at most $\binom{n}{\lfloor n/2\rfloor}$. In Sperner's theorem, $|A\setminus B|=0$ is forbidden.

What can we say if we instead prescribe all possible pairwise intersection sizes? Formally, for a subset $L$ of non-negative integers, a family $\mathcal{F}$ is $L$-intersecting if for any distinct $F,F'\in\mathcal{F}$, $|F\cap F'|\in L$. Results of this type date back to the work of Fisher~\cite{1940Fisher}, who considered the $L$-intersecting problem when $L$ consists of a single element. More precisely, Fisher~\cite{1940Fisher} proved that a uniform family with this property cannot contain more sets than the size of its underlying set. Later, this problem was extensively studied via linear algebra methods by e.g.,~Frankl-Wilson~\cite{1981FranklWilson}, Snevily~\cite{2003Snevily}, Ray-Chaudhuri-Wilson~\cite{1975Ray} and Alon-Babai-Suzuki~\cite{1991AlonBabaiSuzuki}. We also refer the interested readers to the recent survey~\cite{2019book}. 


The following fundamental result was proved by Ray-Chaudhuri and Wilson~\cite{1975Ray} in 1975. 

\begin{theorem}[Ray-Chaudhuri-Wilson~\cite{1975Ray}]\label{thm:RCW}
    Let $L=\{\ell_{1},\ell_{2},\ldots,\ell_{s}\}$ be a set of $s$ non-negative integers and $k$ be a positive integer. Let $\mathcal{F}$ be a $k$-uniform $L$-intersecting family, then $|\mathcal{F}|\le\binom{n}{s}$.
\end{theorem}

Snevily~\cite{1995SnevilyJCD} proposed the following conjecture in 1995.

\begin{conj}[Snevily~\cite{1995SnevilyJCD}]\label{conj:main}
    Let $L=\{\ell_{1},\ell_{2},\ldots,\ell_{s}\}$ be a set of $s$ non-negative integers and $K=\{k_{1},k_{2},\ldots,k_{r}\}$ be a set of $r$ positive integers with $\max{\ell_{i}}<\min{k_{j}}$. Let $\mathcal{F}\subseteq \bigcup_{i=1}^{r}\binom{[n]}{k_{i}}$ be an $L$-intersecting family, then $|\mathcal{F}|\leq\binom{n}{s}$.
\end{conj}

This conjecture, if true, would be a generalization of the Ray-Chaudhuri-Wilson theorem. Indeed, Theorem~\ref{thm:RCW} shows that Conjecture~\ref{conj:main} is true when $|K|=1$. Using the multilinear polynomial methods, Alon, Babai and Suzuki~\cite{1991AlonBabaiSuzuki} proved that families as above satisfy $|\mathcal{F}|\leq\binom{n}{s}+\cdots+\binom{n}{s-r+1}$ and Snevily~\cite{1994JCTASnevily} showed that $|\mathcal{F}|\leq\binom{n-1}{s}+\cdots+\binom{n-1}{0}$. Both of these bounds verify Conjecture~\ref{conj:main} asymptotically, that is, $|\mathcal{F}|\le (1+o(1))\binom{n}{s}$ when $n\rightarrow\infty$. 

In support of Conjecture~\ref{conj:main}, Snevily~\cite{1995SnevilyJCD} proved the special case when $L=\{0,1,\ldots,s-1\}$, and for the general case that $|\mathcal{F}|\leq\sum_{i=s-2r+1}^{s}\binom{n-1}{i}$, implying that $|\mathcal{F}|\leq \binom{n}{s}+\binom{n}{s-1}$ for large $n$. For more results related to Conjecture~\ref{conj:main}, we refer the readers to~\cite{2009JCTAChen, 2019book, 2015EUJC, 2007EUJC, 2018DmWang} and the reference therein.

Observe that, if a set system $\mathcal{F}\subseteq 2^{[n]}$ satisfies the conditions in Conjecture~\ref{conj:main}, then $\mathcal{F}$ is Sperner. By LYM inequality (Theorem~\ref{thm:LYM}), It is clear that Conjecture~\ref{conj:main} is true when $n<2s-1$. We propose a stronger conjecture for Sperner systems as follows.

\begin{conj}\label{conj:new}
    Let $n,s$ be two integers with $n\ge 2s-1$, $L=\{\ell_{1},\ell_{2},\ldots,\ell_{s}\}$ be a set of $s$ non-negative integers and $\mathcal{F}\subseteq 2^{[n]}$ be an $L$-intersecting Sperner system, then $|\mathcal{F}|\leq\binom{n}{s}$.
\end{conj}

Clearly, Conjecture~\ref{conj:new} implies Conjecture~\ref{conj:main}. Our first result verifies Conjecture~\ref{conj:new} when $0\in L$ for reasonably large $n$.

\begin{theorem}\label{thm :0}
    Let $L=\{\ell_{1},\ell_{2},\ldots,\ell_{s}\}$ be a set of $s$ non-negative integers and $\mathcal{F}\subseteq 2^{[n]}$ be an $L$-intersecting Sperner family, we have $|\mathcal{F}| \le \sum_{i=0}^{s}\binom{n-1}{i}$. In particular, if $0\in L$ and $n\ge 3s^{2}$, then $|\mathcal{F}|\leq\binom{n}{s}$, with equality holds if and only if $L=\{0,1,2,\ldots,s-1\}$ and $\mathcal{F} = \binom{[n]}{s}$.
\end{theorem}
Note that when $L=\{0\}$, $\mathcal{F}=\{\{1\},\{2\},\ldots,\{n\}\}$ is a $\{0\}$-intersecting Sperner family, showing that the above upper bound $|\mathcal{F}| \le \sum_{i=0}^{s}\binom{n-1}{i}$ cannot be improved in general.

What if $0\notin L$? We say a family $\mathcal{F}=\{F_{1},F_{2},\ldots,F_{m}\}\subseteq 2^{n}$ has \emph{rank} $t$ if $\max_{i\in [m]}{|F_{i}|}=t$, and say $\mathcal{F}$ is \emph{trivially} $\ell$-intersecting if there is a set $A$ with size $\ell$ such that $A\subseteq\bigcap_{i=1}^{m}F_{i}$. We obtain better bounds when $0\notin L$ and $\mathcal{F}$ has relatively small rank.
\begin{theorem}\label{thm:large}
    Let $L=\{\ell_{1},\ell_{2},\ldots,\ell_{s}\}$ be a set of $s$ positive integers with $\ell_{1}<\ell_{2}<\cdots<\ell_{s}$ and $\mathcal{F}\subseteq 2^{[n]}$ be an $L$-intersecting Sperner family with rank $M$. Then either $\mathcal{F}$ is trivially $\ell_1$-intersecting and then $|\mathcal{F}|\le \sum_{i=0}^s{\binom{n-2}{i}}={\binom{n}{s}}-\Omega_{s}(n^{s-1})$, or $|\mathcal{F}|=O_{M,s}(n^{s-1})$.
\end{theorem}

Combining Theorems~\ref{thm :0} and~\ref{thm:large}, we see that Conjecture~\ref{conj:main} is true when $n$ is sufficiently large.

\begin{cor}\label{cor:large}
    Let $L=\{\ell_{1},\ell_{2},\ldots,\ell_{s}\}$ be a set of $s$ non-negative integers and let $K=\{k_{1},k_{2},\ldots,k_{r}\}$ be a set of $r$ positive integers with $\max{\ell_{i}}<\min{k_{j}}$. There exists $n_0$ such that for any $n\ge n_0$ and $L$-intersecting family $\mathcal{F}\subseteq \bigcup_{i=1}^{r}\binom{[n]}{k_{i}}$, $|\mathcal{F}|\leq\binom{n}{s}$.   
\end{cor}

\subsection{Restricted symmetric/set-wise differences}
For set systems with bounded symmetric differences, Kleitman~\cite{1966Kleitman} proved the following classical result: for integers $n>d$ and a family $\mathcal{F}\subseteq 2^{[n]}$ with $|A\triangle B| \le d$ for any $A,B \in \mathcal{F}$, 
\begin{equation*}
    |\mathcal{F}|\le 
    \begin{cases}
        \sum\limits_{i=0}^{k}\binom{n}{i} &  d=2k; \\ 
   2\sum\limits_{i=0}^{k}\binom{n-1}{i} &  d=2k+1.
    \end{cases}
\end{equation*}
The bounds in both cases above are optimal. For the even case, the upper bound can be attained by the radius-$k$ Hamming ball $\mathcal{K}(n,k) := \{ F:F\subseteq[n], |F|\le k \}$, and for the odd case, consider the family $\mathcal{K}_y(n,k) := \{ F:F\subseteq[n], |F\setminus\{y\}|\le k \}$. The original proof of Kleitman is combinatorial. Recently, Huang, Klurman and Pohoata~\cite{2020Huang} provided an algebraic proof via the Cvetkovi\'{c} bound~\cite{1972Cvetkovic} and extended the results when the allowed symmetric differences lie in a set of consecutive integers.

For a family $\mathcal{F}\subseteq 2^{[n]}$ and a subset $S\subseteq [n]$, we define the \emph{translate} of $\mathcal{F}$ by $S$ as $\mathcal{F}\triangle S:=\{F\triangle S:F\in\mathcal{F}\}$. Recently, Frankl~\cite{2017CPCFrankl} proved the following stability result for Kleitman's theorem.
\begin{theorem}[Frankl~\cite{2017CPCFrankl}]\label{thm:StabilityFrankl}
     Let $n\ge d+2$, $d\geqslant 0$ and let $\mathcal{F}\subseteq 2^{[n]}$ with $|A\triangle B| \le d$ for any $A,B \in \mathcal{F}$. 
     \begin{enumerate}
         \item If $d=2k$ and $\mathcal{F}$ is not contained in any translate of $\mathcal{K}(n,k)$, then $|\mathcal{F}|\leqslant \sum_{i=0}^{k}\binom{n}{i}-\binom{n-k-1}{k}+1$.
         \item If $d=2k+1$ and $\mathcal{F}$ is not contained in any translate of $\mathcal{K}_{y}(n,k)$, then $|\mathcal{F}| \le 2\sum_{i=0}^{k}\binom{n-1}{i} - \binom{n-k-2}{k}+1$.
     \end{enumerate}
\end{theorem}

Frankl's proof is combinatorial and relatively involved, making use of his earlier stability result for Katona theorem~\cite{2017JCTBFrankl}. We prove a finer stability result for Kleitman's theorem for the odd case. 


\begin{theorem}\label{thm:SymmetricStability}
        Let $\mathcal{F} \subseteq 2^{[n]}$ be a family of subsets of $[n]$ with $|A\triangle B| \le 2k+1$ for any $A,B \in \mathcal{F}$, then
either $|\mathcal{F}| \le 2\sum_{i=0}^{k}\binom{n}{i} - 2\binom{n-5k-1}{k}$, or $\mathcal{F}$ is contained in some translate of $\mathcal{K}(n,k+1)$. Furthermore, if $\mathcal{F}$ is contained in some translate of $\mathcal{K}(n,k+1)$, then either $\mathcal{F}$ is contained in some translate of $\mathcal{K}_{y}(n,k)$, or $|\mathcal{F}| \le 2\sum_{i=0}^{k}\binom{n-1}{i} - \binom{n-k-2}{k}+1$.
\end{theorem}
Note that if $\mathcal{F}$ is not contained in any translate of $\mathcal{K}(n,k+1)$, our upper bound $|\mathcal{F}| \le 2\sum_{i=0}^{k}\binom{n}{i} - 2\binom{n-5k-1}{k}$ is better than the bound $|\mathcal{F}| \le 2\sum_{i=0}^{k}\binom{n-1}{i} - \binom{n-k-2}{k}+1$ in Theorem~\ref{thm:StabilityFrankl} when $n\ge 10k^{2}$. We remark that our method does also apply to the even case. However, it yields a slighly weaker bound $|\mathcal{F}| \le \sum_{i=0}^{k}\binom{n+1}{i} - \binom{n-5k}{k}$ when $\mathcal{F}$ is not a translate of $\mathcal{K}(n,k)$.


Another type of well-studied restricted intersection problem is to bound the set-wise differences. In 1983, Katona~\cite{1983Katona} asked if we require $|A\setminus B|\le k$ for any distinct $A,B\in\mathcal{F}$, then what is the largest such $\mathcal{F}$? Frankl~\cite{1985Frankl} proved that
for a set $L$ of non-negative integers, given a family $\mathcal{F}\subseteq 2^{[n]}$ so that for any distinct $A,B\in\mathcal{F}$, $|A\setminus B|\in L$, then $|\mathcal{F}|\le \sum_{i=0}^{|L|}\binom{n}{i}$. 

Our method also yields stability result for families with bounded set-wise differences as follows.
\begin{theorem}\label{thm:DifferenceStability}
Let $n,t,k$ be positive integers with $k<t\le\frac{k+n}{2}$ and $\mathcal{F}=\{F_1,F_2,\ldots,F_m\} \subseteq 2^{[n]}$ be a family with rank $t$ and $|F_i\setminus F_j| \le k$ for any $F_i,F_j \in \mathcal{F}$, then either $\mathcal{F}$ is trivially $(t-k)$-intersecting and $|\mathcal{F}|\le \sum_{i=0}^k\binom{n-(t-k)}{i}$, or $|\mathcal{F}| \le\sum_{i=0}^{k} \binom{n}{i} -\binom{n-t-2k}{k}$. In the latter case, if $n$ is sufficiently large in terms of $t$ and $k$, then $|\mathcal{F}|\le (1-c)\binom{n-(t-k)}{k},$ for some constant $c>0$.
\end{theorem}
We remark that the upper bound on the rank $t\le \frac{k+n}{2}$ is not restrictive. Indeed, for any sets $A,B$, $|A\setminus B|=|B^c\setminus A^c|$. Thus, when $t>\frac{k+n}{2}$, we can consider the family consisting of \emph{complement} $F^c$ of all sets $F\in \mathcal{F}$ instead.

For families with set-wise differences lying in $L$, where $|L|=k$, Frankl~\cite{1985Frankl} conjectured that if the family is Sperner, then the upper bound $|\mathcal{F}|\le \sum_{i=0}^{k}\binom{n}{i}$ can be improved to $\binom{n}{k}$. As a corollary, we show that this conjecture holds for $L=\{1,2,\ldots,k\}$ and large $n$. The following corollary can also be deduced from the main results in~\cite{2017Frankl}.
\begin{cor}\label{coro:main}
     Given integer $k$ and Sperner family $\mathcal{F}\subseteq 2^{[n]}$ with $|A\setminus B| \le k$ for any $A,B \in \mathcal{F}$, if $n$ is sufficiently large, then $|\mathcal{F}| \le \binom{n}{k}$, with equality holds if and only if $\mathcal{F} = \binom{[n]}{k}$.   
\end{cor}

\subsection{Our method}
Linear algebra method has proven to be a powerful tool in extremal set theory, see e.g.,~\cite{1991AlonBabaiSuzuki,1988Babai, 1983Bannai, 1984Blokhuis, 2009JCTAChen, 2007EUJC, 2007JACMubayi, 1975Ray,2003Snevily}. Given a family $\C F$, in this approach we associate members of $\C F$ with multilinear polynomials that are linearly independent. We can then bound the cardinality of the family by the dimension of the underlying space. Improvements can often be made by cleverly adding an extra set of polynomials, which together with the original ones are still linearly independent. 

We introduce a new twist to this classical method, utilizing the non-shadows of the family to find a large collection of polynomials that can be added. There are several advantages of our variation. First of all, the linear independence can usually be verified easily. A more important feature is that, as the new polynomials are associated with non-shadows of $\C F$, we can gather additional structural information of $\C F$, which not only offers optimal bounds in many scenarios but also provides a stability result as shown in Theorems~\ref{thm:SymmetricStability} and~\ref{thm:DifferenceStability}.



\medskip

\noindent\textbf{Organization.} The rest of this paper is organized as follows. In Section~\ref{sec:pre} we will collect some useful tools and give a short proof of a special case of the Katona intersection theorem (Theorem~\ref{thm:shadow}) to illustrate our method. We will prove the stability results, Theorems~\ref{thm:SymmetricStability} and~\ref{thm:DifferenceStability}, in Section~\ref{sec:Stability}. The proofs of Theorems~\ref{thm :0} and~\ref{thm:large} will be presented in Section~\ref{sec:Snevily}.\\

{\bf \noindent Notations.} In this paper, we usually regard $[n]=\{1,2,\ldots,n\}$ as the ground set. For a set $A\subseteq [n]$, we write $A^{c}$ for its complement, that is, $A^{c}=[n]\setminus A$. For a subset $A\subseteq [n]$, we will use $\binom{A}{k}$ to denote the family of all subsets of $A$ with size $k$ and $\binom{A}{\leq k}$ to denote the family of all subsets of $A$ with size at most $k$. For a pair of vectors $\boldsymbol{x}=\{x_{1},x_{2},\ldots,x_{n}\}$ and $\boldsymbol{y}=\{y_{1},y_{2},\ldots,y_{n}\}$, we define their inner product as $\boldsymbol{x}\cdot\boldsymbol{y}=\sum_{i=1}^{n}x_{i}y_{i}$. We use $\boldsymbol{1}_n$ to denote the length-$n$ all-ones vector $(1,1,\ldots,1)$ and omit the subscript when the dimension is clear. Given a polynomial $f$ in $n$ variables $x_{1},x_{2},\ldots,x_{n}$, define its \emph{multilinear reduction} $\tilde{f}$ to be the polynomial obtained from $f$ by replacing each $x_{i}^{t}$ term by $x_{i}$ for any positive integer $t$ and $1\le i\le n$. For a family $\mathcal{F}\subseteq 2^{[n]}$, we denote the \emph{$k$-shadow} of $\mathcal{F}$ as $\partial_{k}{\mathcal{F}}:=\{T\in\binom{[n]}{k}:T\subseteq F\ \text{for\ some\ }F\in\mathcal{F}\}$.

\section{Preliminaries}\label{sec:pre}
\subsection{Some useful tools}
Hilton and Milner~\cite{1967Hilton} showed the following result, which states that the size of non-trivial intersecting families is noticeably smaller than that of the trivial ones.

\begin{theorem}[Hilton-Milner~\cite{1967Hilton}]\label{thm:HiltonMilner}
Let $n,k$ be positive integers with $n>2k$. If $\mathcal{F}\subseteq \binom{[n]}{k}$ is an intersecting family and $\bigcap_{F\in\mathcal{F}}F=\emptyset$, then we have
\begin{equation*}
  |\mathcal{F}| \le \binom{n-1}{k-1} - \binom{n-k-1}{k-1}+1.   
\end{equation*}
\end{theorem}

For Sperner families, the Lubell–Yamamoto–Meshalkin inequality, discovered by Bollob\'{a}s~\cite{1965Bollobas}, Lubell~\cite{1966Lubell}, Me\v{s}alkin~\cite{1963Me} and Yamamoto~\cite{1954Ya} independently, is useful.
\begin{theorem}[LYM inequality]\label{thm:LYM}
     Let $\mathcal{F}\subseteq 2^{[n]}$ be a Sperner family of subsets of $[n]$, then 
     \begin{equation*}
         \sum\limits_{A\in\mathcal{F}}\frac{1}{\binom{n}{|A|}}\le 1.
     \end{equation*}
\end{theorem}

We also need the triangular criterion when we want to prove a sequence of polynomials are linearly independent.

\begin{prop}\label{prop:triangular}
    Let $f_{1},f_{2},\ldots,f_{m}$ be functions in a linear space. If $v^{(1)},v^{(2)},\ldots,v^{(m)}$ are vectors such that $f_{i}(v^{(i)})\neq 0$ for $1\le i\le m$ and $f_{i}(v^{(j)})=0$ for $i>j$, then $f_{1},f_{2},\ldots,f_{m}$ are linearly independent.
\end{prop} 

\subsection{Warm up: Katona intersection theorem}
As a warm-up and illustration of our method, we will give a short proof of a special case of the Katona intersection theorem. 
\begin{theorem}[Katona~\cite{1964Katona}]\label{thm:shadow}
    Suppose $\mathcal{F} \subseteq \binom{[n]}{k+1}$ is an intersecting family. Then $|\partial_{k}{\mathcal{F}}| \ge |\mathcal{F}|$.
\end{theorem}
\begin{proof}[Proof of Theorem~\ref{thm:shadow}]
 Suppose $\mathcal{F}=\{F_{1},F_{2},\ldots,F_{m}\} \subseteq \binom{[n]}{k+1}$ is an intersecting family. Let $\boldsymbol a_{i},\boldsymbol b_{i}$ be the characteristic vectors of $F_i, F_i^c$, respectively. For $\boldsymbol x=(x_1,\ldots,x_n)$ and $i\in [m]$, we define $g_i(\boldsymbol x) =\prod_{j=1}^{k} (\boldsymbol x \cdot \boldsymbol b_i -j)$ and let $f_i=\tilde{g}_{i}$ be the multilinear reduction of $g_i$. Note that $\ker(f_{i})$ corresponds to characteristic vectors of subsets $F\subseteq [n]$ with $|F\setminus F_{i}|\in [k]$,  

Let $\{S_1,\ldots,S_r\} = \binom{[n]}{\le k-1}$ with $|S_i|\le |S_j|$ for $i< j$, so $r = \sum_{i=0}^{k-1}\binom{n}{i}$.
For each $S_{i}$, let $f_{S_i}$ be the multilinear reduction of $g_{S_i}(\boldsymbol x):=(\boldsymbol{x}\cdot \boldsymbol{1} -k-1)\cdot\prod_{\ell\in S_i }x_{\ell}$, and $\boldsymbol a_{S_i}$ be the characteristic vector of $S_i$. Let $\{T_1, T_2,\ldots,T_h\} =\binom{[n]}{k}\setminus \partial_{k}{\mathcal{F}}$ be the collection of all $k$-sets that are not $k$-shadow of $\mathcal{F}$. For each $T_{i}$, we define $f_{T_i}(\boldsymbol{x}):= \prod_{j\in T_i }x_j$, and  $\boldsymbol a_{T_i}$ to be the characteristic vector of $T_i$. Here we can see that $\ker(f_{T_{i}})$ corresponds to characteristic vectors of subsets of $[n]$ which do not contain $T_{i}$, and $\ker(g_{S_{i}})$ corresponds to characteristic vectors of subsets of $[n]$ which contain $k+1$ elements or do not contain $S_{i}$.

\begin{claim}\label{claim:S7}
   The polynomials $\{f_1,\ldots,f_m,f_{S_1},\ldots,f_{S_r},f_{T_1},\ldots,f_{T_h} \}$ are linearly independent. 
\end{claim} 
\begin{poc}
    We shall show that these polynomials satisfy the triangular criterion in Proposition~\ref{prop:triangular}.
First note that $f_i(\boldsymbol a_i) =(-1)^{k}k! \ne 0$ as $\boldsymbol a_{i}\cdot \boldsymbol b_{i}=0$ for each $i$. For $i\ne j$, we have $f_i(
\boldsymbol a_j)=0$ as $\boldsymbol a_{j}\cdot \boldsymbol b_{i}=|F_{j}\cap F_{i}^{c}|=|F_j\setminus F_i| \in [k]$.
Next, $f_{S_i}(\boldsymbol a_{S_i}) = (|S_i|-k-1) \ne 0$. Since $|F_i|=k+1$, we know that $f_{S_i}(\boldsymbol a_j) =0$. By definition, if $i<j$, we have $|S_i| \le |S_j|$ and $S_{j}\setminus S_{i}\ne \emptyset$. Thus for any $i<j$, we have $f_{S_j}(\boldsymbol a_{S_i}) = 0$. For $f_{T_{i}}$'s, we have $f_{T_i}(\boldsymbol a_{T_i}) = 1 \ne 0$, and for any $T\in \{T_1,\ldots,T_{i-1},T_{i+1}\ldots,T_h,S_1,S_2,\ldots,S_{r}\}$, we can see $|T| \le |T_i|, T\ne T_i$ and $T_{i}\setminus T\ne \emptyset$, so we have $f_{T_i}(\boldsymbol{a}_{T}) =0$. Lastly, since $T_i \not\subseteq F_j$ for any $i,j$, we have $f_{T_i}(\boldsymbol a_{j}) = 0$. Hence, Proposition~\ref{prop:triangular} implies that $\{f_1,\ldots,f_m,f_{S_1},\ldots,f_{S_r},f_{T_1},\ldots,f_{T_h} \}$ are linearly independent.
\end{poc}
Since all of these polynomials have degree at most $k$, we have $|\mathcal{F}|+r+h\le \sum _{i=0}^{k}\binom{n}{i}$. As $r = \sum_{i=0}^{k-1}\binom{n}{i}$ and $h=\binom{n}{k} -|\partial_{k}{\mathcal{F}}|$, we have $|\mathcal{F}| \le |\partial_{k}{\mathcal{F}}|$ as desired.

\end{proof}

\section{Stability}\label{sec:Stability}
\subsection{Families with bounded symmetric differences}
In this section, we will prove Theorem~\ref{thm:SymmetricStability}.
Let $\mathcal{F}_o$ consist of all sets of odd size in $\mathcal{F}$ and $\mathcal{F}_e$ consist of all sets of even size in $\mathcal{F}$. Observe that for any subsets $X,Y,Z \subseteq [n]$, $(X\triangle Y)\triangle(X\triangle Z)= Y\triangle Z$. Hence, by translating all sets in $\C F$ by some appropriate $U\in \C F$, we may assume $\emptyset\in\mathcal{F}$ and $|F|\le 2k+1$ for any $F\in \mathcal{F}$. Then without loss of generality, we can further assume that $|\mathcal{F}_o|\le|\mathcal{F}_e|$.

Let $\mathcal{F}_e =\{F_1,\ldots, F_m\}$ and $\boldsymbol a_{i},\boldsymbol b_{i}$ be the characteristic vectors of $F_i, F_i^c$, respectively. For $\boldsymbol x=(x_1,\ldots,x_n)$ and $i\in [m]$, we define $g_i(\boldsymbol x) =\prod_{\ell=1}^{k} (\boldsymbol x \cdot \boldsymbol b_i + (\boldsymbol{1}-\boldsymbol x)\cdot \boldsymbol a_i-2\ell)$ and $f_i=\tilde{g}_{i}$. It is easy to see that $\ker(g_{i})$ corresponds to characteristic vectors of subsets $F\subseteq [n]$ with $|F_{i}\triangle F|\in\{2,4,\ldots,2k\}$.

By definition, $f_i(\boldsymbol a_i) = (-2)^k k! \ne 0$ as $\boldsymbol a_{i}\cdot \boldsymbol b_{i}=0$ for each $i$. As each set in $\mathcal{F}_{e}$ has even size, $\boldsymbol a_{i}\cdot \boldsymbol b_{j} + \boldsymbol b_i \cdot \boldsymbol a_j=|F_i\triangle F_j| \in \{2,4,\ldots,2k\}$ for any $i\ne j$, so we have $f_i(
\boldsymbol a_j)=0$. By Proposition~\ref{prop:triangular}, $\{f_i\}_{i=1}^{m}$ are linearly independent, which implies $m\le \sum_{i=0}^{k} \binom{n}{i}$ since the polynomials $\{f_i\}_{i=1}^{m}$ lie in the space of $n$-variate polynomials of degree at most $k$.

Let $\mathcal{S}= \binom{[n]}{k}\setminus\partial_{k}{\mathcal{F}}_{e}:= \{T\in \binom{[n]}{k} : T\not\subseteq F_i,\text{ for all } i\in[m]\}$. For each subset $T\in \mathcal{S}$, we define a polynomial $h_T(\boldsymbol x) =\prod_{i\in T} x_i$ and let $\boldsymbol a_T$ be the characteristic vector of $T$. Similarly, $\ker(h_{T})$ corresponds to characteristic vectors of subsets of $[n]$ which do not contain $T$ as a subset. The following claim will establish an upper bound $|\mathcal{F}_{e}|\leq \sum_{i=0}^{k} \binom{n}{i}-|\mathcal{S}|$.

\medskip
\begin{claim}\label{claim:S1}
   The polynomials $\{h_T\}_{T\in \mathcal{S}}$ and $\{f_i\}_{i=1}^{m}$ are linearly independent. 
\end{claim}

\begin{poc}
For any $T\in \mathcal{S}$ and $i\in [m]$, since $T\not\subseteq F_i$, we have $h_T(\boldsymbol a_i) = 0$. For any $T,S\in \mathcal{S}$, it is easy to check that $ h_T(a_S)\neq 0$ if and only if $T=S$. By Proposition~\ref{prop:triangular}, $\{h_T\}_{T\in \mathcal{S}}$ and $\{f_i\}_{i=1}^{m}$ are linearly independent. 
\end{poc}
If $|\mathcal{S}|\ge \binom{n-5k-1}{k}$, then we have
\begin{equation*}
    |\mathcal{F}|\le 2|\mathcal{F}_e|\le 2\big(\sum_{i=0}^{k}\binom{n}{i} -|\mathcal{S}|\big)  \le 2\sum_{i=0}^{k}\binom{n}{i} - 2\binom{n-5k-1}{k},
\end{equation*}
yielding the first alternative. 

We may then assume $|\mathcal{S}|<\binom{n-5k-1}{k}$.
Relabelling $\mathcal{F}_{e}$ if necessary, we may assume $2k+1\ge |F_1|\ge |F_2|\ge \cdots \ge |F_m|$. As $|F_1|$ is maximum, we have $|F_i\setminus F_1|\le |F_1\setminus F_i|$ for any $1\le i\le m$. Since $|F_1\triangle F_i| \le 2k+1$ and both of $|F_1|$ and $|F_i|$ are even, we have $|F_i\setminus F_1|\le k$ and $|F_1|\le 2k$. We can then draw the following claim.

\begin{claim}\label{claim:S2}
If $|F_i\setminus F_1|= k$, then we have $|F_1\setminus F_i|=k$, and $|F_1|=|F_i|$. 
\end{claim} 

Since $|\mathcal{S}|< \binom{n-5k-1}{k}$ and $n-|F_1|> n-5k-1$, there exists some $k$-set $R\subseteq F_{1}^{c}$ such that $R\in \partial_{k}{\mathcal{F}_{e}}$. Consequently, there exists some $E\in \mathcal{F}_e$ such that $R\subseteq E$ and $|E\setminus F_1| \ge |R|= k$. In fact, $|E\setminus F_{1}|=k$ as $|E\setminus F_{1}|\le k$. Claim~\ref{claim:S2} then implies that $|F_{1}|=|E|$. Fix such $E$ and let $A: = F_1\cap E$.

\begin{claim}\label{claim:S3}
    For any $F\in \mathcal{F}$, we have $|F \triangle A|\le k+1$. 
\end{claim} 
\begin{poc}
Suppose that there exists $F\in \mathcal{F}$ such that $|F \triangle A|\ge k+2$. Since $|F_1|\le 2k$, $|E\setminus F_1|= k$ and $|F|\le 2k+1$, we have $|F_1\cup E \cup F|\le 5k+1$.
Since $|\mathcal{S}|< \binom{n-5k-1}{k}$, there exists some $k$-set $Q\subseteq [n]\setminus (F_1\cup E\cup F)$ such that $Q\in\partial_{k}{\mathcal{F}}_{e}$. Thus there is some $T\in\mathcal{F}_{e}$ such that $Q\subseteq T$ and $|T\cap ([n]\setminus (F_1\cup E\cup F))|\ge k$. As $|T\setminus F_1| \le k$, we have $T\cap ([n]\setminus (F_1\cup E\cup F)) = T\setminus F_1$. 
Therefore $T\setminus F_1$, $E\setminus A$ and $F_1\setminus A$ are pairwise disjoint sets of size $k$. By Claim~\ref{claim:S2}, $|T\triangle F_1|= 2k$, $|T\triangle E|= 2k$ and $T=A \cup (T\setminus F_1)=A\cup (T\cap([n]\setminus (F_{1}\cup E\cup F)))$. Then we have
\begin{equation*}
    |T\triangle F|= |T\cap ([n]\setminus (F_1\cup E\cup F))|+ |A\triangle F|\ge k+k+2=2k+2,
\end{equation*}
 a contradiction.
\end{poc}

By Claim~\ref{claim:S3}, if $|\mathcal{S}|<\binom{n-5k-1}{k}$, then $\mathcal{F}$ should be contained in a translate of $\mathcal{K}(n,k+1)$. Without loss of generality, assume $\mathcal{F}$ is contained in $\mathcal{K}(n,k+1)$. Let $\mathcal{H} = \{H\in \mathcal{F}, |H| = k+1\}$. Note that $\mathcal{H}$ is an intersecting family, because the symmetric difference of any pair of subsets in $\mathcal{F}$ is bounded by $2k+1$. If $\mathcal{H}$ is a non-trivial intersecting family, by Hilton-Milner Theorem, we have
\begin{equation*}
    |\mathcal{F}|\le \sum_{i=0}^{k}\binom{n}{i}+|\mathcal{H}|\le \sum_{i=0}^{k}\binom{n}{i}+ \binom{n-1}{k}- \binom{n-k-2}{k} +1=2\sum\limits_{i=0}^{k}\binom{n-1}{i} - \binom{n-k-2}{k}+1.
\end{equation*}
Otherwise, there exists some $y\in[n]$ such that $y\in H$ for any $H\in \mathcal{H}$, which implies $\mathcal{F}$ is contained in $\mathcal{K}_{y}(n,k)$. This completes the proof.

\subsection{Families with bounded set-wise differences}
In this section, we prove Theorem~\ref{thm:DifferenceStability} and then apply it to show Corollary~\ref{coro:main}.
\begin{proof}[Proof of Theorem~\ref{thm:DifferenceStability}]
    
Let $\mathcal{F}=\{F_1,F_2,\ldots,F_m\} \subseteq 2^{[n]}$ be a family with rank $t$ and $|F_i\setminus F_j| \le k$ for any $F_i,F_j \in \mathcal{F}$. Without loss of generality, we assume that $|F_j|\ge |F_i|$ for any $i>j$, then we have $|F_1| = t$ and $|F_j\setminus F_i|\ge 1$ for any $i>j$.

Let $\boldsymbol a_{i},\boldsymbol b_{i}$ be the characteristic vectors of $F_i, F_i^c$, respectively. For $\boldsymbol x=(x_1,\ldots,x_n)$ and $i\in [m]$, we define $g_i(\boldsymbol x) =\prod_{j=1}^{k} (\boldsymbol x \cdot \boldsymbol b_i -j)$ and $f_i=\tilde{g}_{i}$. Here $\ker(f_{i})$ corresponds to characteristic vectors of subsets $F\subseteq [n]$ with $|F\setminus F_{i}|\in [k]$.

By definition, $f_i(\boldsymbol a_i) =(-1)^{k} k! \ne 0$ as $\boldsymbol a_{i}\cdot \boldsymbol b_{i}=0$ for each $i$. For $i>j$, we have $f_i(
\boldsymbol a_j)=0$ as $\boldsymbol a_{j}\cdot \boldsymbol b_{i}=|F_j\setminus F_i| \in [k]$. By Proposition~\ref{prop:triangular}, $\{f_i\}_{i=1}^{m}$ are linearly independent, which implies $m\le \sum_{j=0}^{k} \binom{n}{j}$.

Let $\mathcal{S}=\binom{[n]}{k}\setminus\partial_{k}{\mathcal{F}}$. For each subset $T\in \mathcal{S}$, we define a polynomial $h_T(\boldsymbol x) =\prod_{i\in T} x_i$. The following claim can be derived the same way as Claim~\ref{claim:S1}.

\begin{claim}\label{claim:S4}
   The polynomials $\{h_T\}_{T\in \mathcal{S}}$ and $\{f_i\}_{i=1}^{m}$ are linearly independent. 
\end{claim} 

Claim~\ref{claim:S4} implies that $m \le \sum_{i=0}^{k} \binom{n}{i} - |\mathcal{S}|$. Thus, we shall assume $|\mathcal{S}| < \binom{n-t-2k}{k}$, otherwise the second alternative holds.

\begin{claim}\label{claim:S5}
    For any pair of distinct $E,F\in \mathcal{F}$ with $|F|=t,|E \setminus F| =k$, we must have $|E| =t$ and $|F\setminus E| =k$.
\end{claim} 

\begin{poc}
    Note that $\max_{i\in[m]}\{|F_i|\}=t$, thus we have $k\ge |F\setminus E|\ge |E\setminus F| \ge k$ and so $|F\setminus E|=k$. Also, $|E|=|E\cup F|-|F\setminus E|=|E\cup F|-|E\setminus F|=|F|=t$.
\end{poc}

By our assumption, we have $\binom{|F_1^c|}{k}= \binom{n-t}{k} >|\mathcal{S}|$, hence there exists some $F\in \mathcal{F}$ such that $|F\cap F_1^c| =| F\setminus F_1| = k$. By Claim~\ref{claim:S5}, we have $|F|=t$ and $|F\cap F_1| =t-|F\setminus F_{1}|= t-k$. Let $F\cap F_1 =A$. We can further extract more structural properties of $\mathcal{F}$ as follows. 

\begin{claim}\label{claim:S6}
    For any $F' \in \mathcal{F}$, if $|F'\cap F_1^c|=k$, then we have $F'\cap F_1 = A$.
\end{claim}
\begin{poc}
Since $|F_{1}^{c}\setminus (F'\cup F)|\ge n-t-2k$ and $|\mathcal{S}|<\binom{n-t-2k}{k}$, there exist $T\in \binom{F_1^c\setminus (F'\cup F)}{k}$ and $E \in \mathcal{F}$ such that $T\subseteq E$. Since $|T|=k$, we have $E\setminus F'=E\setminus F = T \subseteq F_1^c,$ by Claim~\ref{claim:S5} we have $|E| =t$. Also note that $F\cap E\subseteq F_{1}$, so $|F\cap E| =|F\cap E\cap F_1|=t-k$, which implies $E\cap F_1=F\cap F_1=A$. Applying the same analysis on the (ordered) triple $(F',E,F_1)$ instead of $(E,F,F_1)$, we have $F'\cap F_1=E\cap F_1=A$ as desired.
\end{poc}

Next, we will show that for any $F' \in \mathcal{F}$, we have $A \subseteq F' \cap F_1$, and so $\mathcal{F}$ is trivially $(t-k)$-intersecting with $A\subseteq\bigcap_{F'\in\mathcal{F}}F'$. Suppose there exists some $F'$, such that $A \setminus  (F' \cap F_1) \ne \emptyset$, then since $|F_1^c \setminus F'| \ge n-t-k$, and  $|\mathcal{S}|<\binom{n-t-2k}{k}$, there exist subsets $T\in \binom{F_1^c\setminus F'}{k}$ and $E \in \mathcal{F}$ such that $T\subseteq E$. By Claim~\ref{claim:S6} we know that $E\cap F_1 = A$, so $|E\setminus F'|$ = $|T \cup (A \setminus F')| \ge k+1$, which is a contradiction. 

Then we have that $A\subseteq F_i$ for all $i\in [m]$. Since $|F_i|\le t$, we have $|F_i\setminus A| \le k$, which implies that $m \le \sum_{i=0}^{k}\binom{n-(t-k)}{i}$. 
The proof of Theorem~\ref{thm:DifferenceStability} is finished.
\end{proof}

\begin{proof}[Proof of Corollary~\ref{coro:main}]
     Suppose $\mathcal{F}=\{F_1,F_2,\ldots,F_m\} \subseteq 2^{[n]}$ is a rank $t$ Sperner family such that $|F_i\setminus F_j| \le k$.
    Suppose first that $t>\frac{k+n}{2}$, as $|F_i\setminus F_j| \le k$ for any $F_i,F_j \in \mathcal{F}$, we know $\min_{i\in[m]}\{|F_i|\} \ge t-k$. Define $\mathcal{F}':=\{F_1^c,F_2^c,\ldots,F_m^c\}$. 
    Then $\max_{i\in[m]}\{|F_i^c|\}=n-\min_{i\in[m]}\{|F_i|\} \le n- (t-k)\le \frac{k+n}{2} $ and $|F_i^c\setminus F_j^c|=|F_j\setminus F_i|\le k$ for any $F_i^c,F_j^c \in \mathcal{F}'$, so we can assume $t\le\frac{k+n}{2}$.

    If $t>k$, by Theorem~\ref{thm:DifferenceStability}, we have either 
    \begin{equation*}
        |\mathcal{F}| \le  \sum_{i=0}^{k} \binom{n}{i} -\binom{\frac{n-5k}{2}}{k} < \binom{n}{k},
    \end{equation*}
    when $n\gg k$, or there exists $A$ with $|A|=t-k$ and $A\subseteq F_i$  for any $i\in[m]$, which implies $\{F_i \setminus A\}_{i=1}^{m}$ is a Sperner family with size at most $k$. 
    By Theorem~\ref{thm:LYM}, we have 

    \begin{equation*}
        1\ge \sum_{F\in \mathcal{F}} \frac{1}{\binom{n-(t-k)}{|F\setminus A|}} \ge \sum_{F\in \mathcal{F}} \frac{1}{\binom{n-(t-k)}{k}},
    \end{equation*}
    which implies $|\mathcal{F}| \le \binom{n-(t-k)}{k} <\binom{n}{k}$.

    If $t\le k$. By Theorem~\ref{thm:LYM}, we have 
    $$1\ge \sum_{F\in \mathcal{F}} \frac{1}{\binom{n}{|F|}} \ge \sum_{F\in \mathcal{F}} \frac{1}{\binom{n}{k}},$$ which implies $|\mathcal{F}| \le \binom{n}{k}$, the equality holds if and only if $\mathcal{F} = \binom{[n]}{k}$.
\end{proof}

\section{Sperner families with restricted intersections}\label{sec:Snevily}
In this section, we will prove Theorem~\ref{thm :0} and Theorem~\ref{thm:large}, respectively.
\subsection{Proof of Theorem~\ref{thm :0}}
Let $L=\{\ell_{1},\ell_{2},\ldots,\ell_{s}\}$ be a set of $s$ non-negative integers with $\ell_{1}<\ell_{2}<\cdots<\ell_{s}$. Let $\mathcal{F} = \{F_1,\ldots,F_m\}$ be an $L$-intersecting Sperner family. By relabelling, we can assume there exists some integer $r$ such that $1\in F_i$ if and only if $1\le i\le r$. For each $1\leq i\leq m$, let $\boldsymbol a_{i}$ be the characteristic vector of $F_i$. 
    For $\boldsymbol x=(x_1,\ldots,x_n)$ and $i\in [m]$, we define $g_i(\boldsymbol x) =\prod_{\ell \in L, \ell \ne |F_i|} (\boldsymbol x \cdot \boldsymbol a_i -\ell)$ and $f_i(\boldsymbol x)$ be the multilinear reduction of $g_i(\boldsymbol x)$ with $x_{1}=1$, that is, $f_i(\boldsymbol x)$ is obtained from $g_{i}(\boldsymbol{x})$ by replacing each $x_{i}^{t}$ term by $x_{i}$ for any $2\le i\le n$ and $t\ge 2$, and $x_{1}^{t}$ term by $1$ for any $t\ge 1$.

    Note that $\boldsymbol a_i \cdot \boldsymbol a_j = |F_i\cap F_j|$, and we replace $x_{1}^{t}$ with $1$ in our multilinear reduction to get $f_{i}$, this acts as if $1\in F_{j}$ when we evaluate $f_{i}(\boldsymbol{a}_{j})$ and hence 
    $$f_i(\boldsymbol a_j)= \prod_{\ell \in L, \ell \ne |F_i|} (|F_i\cap (F_j\cup\{1\})| -\ell).$$
Moreover, observe that if $i>j$ and $r\ge i$, then $|F_{i}\cap (F_{j}\cup\{1\})|=|F_{i}\cap F_{j}|$ as $1\in F_{j}$, and if $i>j$ and $r<i$, then $|F_{i}\cap (F_{j}\cup\{1\})|=|F_{i}\cap F_{j}|$ also holds as $1\notin F_{i}$. Therefore, if $i>j$, we have $f_i(\boldsymbol a_j)= \prod_{\ell \in L, \ell \ne |F_i|} (|F_i\cap F_j| -\ell)$. Since $\mathcal{F}$ is a Sperner family, we know $|F_i\cap F_j| \in L\setminus \{|F_i|\}$, which implies that $f_i(\boldsymbol a_j) =0$ for $i>j$. If $i=j$, then $f_i(\boldsymbol a_j)= \prod_{\ell \in L, \ell \ne |F_i|} (|F_i| -\ell) \ne 0$. By Proposition~\ref{prop:triangular}, the polynomials $\{f_i\}_{i=1}^{m}$ are linearly independent, which implies $m\le \sum_{j=0}^{s} \binom{n-1}{j}$ since for $i\in [m]$, $f_i(\boldsymbol x)$ has $n-1$ variables $x_2,x_3,\ldots,x_n$ and degree at most $s$. The first result then follows.

    For the second part, let $\ell_{1}=0\in L$. We will prove $|\mathcal{F}| \le \binom{n}{s}$ by induction on $s$, where $s=|L|$ and $n\ge 3s^{2}$.
    When $s=1$, by the first result, we have $|\mathcal{F}| \le  \sum_{j=0}^{s} \binom{n-1}{j} = \binom{n}{1}$, and it is easy to see that equality holds if and only if $\mathcal{F}={[n]\choose 1}$. 
    
    Now, suppose $s\ge 2$ and the result holds for any positive integer less than $s$. If $L=\{0,1,2,\ldots,s-1\}$, then for any $S\in \binom{[n]}{s}$, there exists at most one set in $\mathcal{F}$ containing $S$. Construct a new family $\mathcal{F'}:=\{F'_{1},\ldots,F'_{m}\}$, where $F'_i$ is an arbitrary subset of $F_i$ with size $s$ if $|F_{i}|> s$, and $F'_i = F_i$ otherwise. Since $\mathcal{F}$ is a Sperner family, $\mathcal{F'}$ is also a Sperner family with $|F'_i|\le s$. Then by Theorem~\ref{thm:LYM}, we have $m\le \binom{n}{s}$. It is easy to see that the equality holds if and only if $\mathcal{F} = \mathcal{F}' = \binom{[n]}{s}$. Next, we assume that $L\ne \{0,1,2,\ldots,s-1\}$.

    Let $\mathcal{H}= \{F : F \in \mathcal{F}, |F|\le s\}$ and $L' = L\cap \{0,1,2,\ldots, s-1\}$. Then $\mathcal{H}$ is an $L'$-intersecting Sperner family, with $0\in L'$ and $|L'|\le s-1$. By the induction hypothesis, we have $|\mathcal{H}|\le \binom{n}{s-1}$.

    For each $x\in [n]$, let $\mathcal{F}_x=\{F\in \mathcal{F}: x\in F \}$, and $L\setminus\{0\} = \{\ell_2,\ell_3,\ldots,\ell_s\}$. Then $\mathcal{F}_x$ is an $L\setminus\{0\}$-intersecting Sperner family with $|L\setminus\{0\}|\le s-1$, and by the first part $|\mathcal{F}_x| \le \sum_{i=0}^{s-1}\binom{n-1}{i}$. Then via double counting, we have
    \begin{equation*}
        (m-|\mathcal{H}|)(s+1) \le \sum_{x\in[n]}|\mathcal{F}_x| \le n \sum_{i=0}^{s-1}\binom{n-1}{i}.
    \end{equation*}    
    Since $n\ge 3s^2, s\ge 2$, we have $\sum_{i=0}^{s-2} \binom{n-1}{i} \le \binom{n-1}{s-2}\cdot \sum_{i=0}^{s-2} 2^{-i} \le 2 \binom{n-1}{s-2} $. Therefore, we have
    \begin{align*}
        m&\le |\mathcal{H}|+\frac{n}{s+1}\sum_{i=0}^{s-1}\binom{n-1}{i}
        \le \binom{n}{s-1} +\frac{n}{s+1} \binom{n-1}{s-1} +  \frac{n}{s+1}\sum_{i=0}^{s-2} \binom{n-1}{i} \\
        &\le \binom{n}{s-1} +\frac{s}{s+1} \binom{n}{s} +  \frac{2n}{s+1}\binom{n-1}{s-2} = \binom{n}{s-1} +\frac{s}{s+1} \binom{n}{s} +  \frac{2(s-1)}{s+1}\binom{n}{s-1}\\ &= \bigg(\frac{3s-1}{s+1}\cdot\frac{s}{n-s+1} + \frac{s}{s+1}\bigg)\binom{n}{s} \le  \bigg(\frac{1}{s+1}\cdot\frac{3s^2-s}{3s^2-s+1} + \frac{s}{s+1}\bigg)\binom{n}{s} < \binom{n}{s}.
    \end{align*}

\subsection{Proof of Theorem~\ref{thm:large}}
 Let $T\in\binom{[n]}{\ell_{1}}$ be the set satisfying that $|\{F\in\mathcal{F}: T\subseteq F\}|$ is the largest among all sets in $\binom{[n]}{\ell_{1}}$. Let $\mathcal{H} = \{ F\in \mathcal{F}: T\subseteq F\}$. If $\mathcal{H}=\mathcal{F}$, then $\C F$ is trivially $\ell_1$-intersecting and by Theorem~\ref{thm :0}, as $|T|=\ell_{1}\ge 1$, we have $\{F\setminus T :F\in\mathcal{F}\}\subseteq 2^{[n]\setminus T}$ and so
    \begin{equation*}
        |\mathcal{F}|=\big|\{F\setminus T :F\in\mathcal{F}\}\big| \le \sum_{i=0}^{s} \binom{n-2}{i} = \binom{n}{s} -\binom{n-1}{s-1} +\sum_{i=0}^{s-2} \binom{n-2}{i} = \binom{n}{s}-\Omega_{s}(n^{s-1}).
    \end{equation*}

    We may then assume $\C H\neq \C F$ and so there exists some $E \in \mathcal{F}$ with $T \not\subseteq E$. For any $F\in \mathcal{H}$ we have $|F\cap E| \ge \ell_1$, so $(E\setminus T)\cap F\ne \emptyset$, for otherwise $T\subseteq E$. Consequently, for each $x\in E\setminus T$, letting $\mathcal{F}_x=\{ F: F\in\mathcal{H}, x\in F \}$, we have $\mathcal{H}\subseteq \bigcup_{x\in E\setminus T} \mathcal{F}_x$. Since for any sets $A,B \in \mathcal{F}_x$, $\{x\}\cup T
    \subseteq A\cap B$, $F_x$ is an $\{\ell_2,\ell_3,\ldots,\ell_s\}$-intersecting Sperner family. Then, by Theorem~\ref{thm :0}, we have $|\mathcal{F}_x| \le \sum_{i=0}^{s-1}\binom{n-1}{i}$ and  $|\mathcal{H}| \le M \sum_{i=0}^{s-1}\binom{n-1}{i}$. 

    Take an arbitrary set $X\in\mathcal{F}$. For any set $F \in \mathcal{F}$, we have $|X\cap F| \ge \ell_1$, thus we can partition the whole family $\mathcal{F}$ into subfamilies $\mathcal{H}_{Y}$ according to which $\ell_{1}$-subset $Y\in \binom{X}{\ell_{1}}$ its members contain. Note that for each such $\C H_Y$, $|\mathcal{H}_{Y}|\le |\mathcal{H}|$. As $|X|\le M$, we have $|\mathcal{F}| \le  \binom{M}{\ell_1}|\mathcal{H}| =O_{M,s}(n^{s-1})$.

    \section*{Acknowledgement}
  We would like to express our gratitude to the anonymous reviewer for the detailed and constructive comments which are helpful to the improvement of the technical presentation of this paper. We also thank Yongtao Li and Tuan Tran for helpful discussions.

\bibliographystyle{abbrv}
\bibliography{TightSperner}
\end{document}